\documentclass[a4paper,12pt]{article}
\usepackage{amsmath,amssymb,amsthm}
\newcommand{\C}{{\mathbb C}}
\newcommand{\K}{{\mathbb K}}
\newcommand{\N}{{\mathbb N}}
\newcommand{\R}{{\mathbb R}}

\newcommand{\abs}[2][\empty]{\ifx#1\empty\left|#2\right|%
\else#1\vert #2 #1\vert\fi}
\newcommand{\caninf}{\rho}
\newcommand{\Cnt}[1][]{{\cal C}^{#1}}
\newcommand{\clos}[1]{\overline{#1}}
\newcommand{\co}[1]{{#1}^{c}}
\newcommand{\defstyle}[1]{{\bf #1}}
\newcommand{\eps}{\varepsilon}
\newcommand{\Filter}{\mathcal F}
\newcommand{\ideal}{\unlhd}
\newcommand{\idealproper}{\lhd}
\renewcommand{\implies}{\Rightarrow}
\newcommand{\inter}[1]{{#1}^\circ}
\newcommand{\Max}{\mathcal M}
\newcommand{\pure}{m}
\newcommand{\rad}[1]{\sqrt{#1}}
\newcommand{\radpart}[1]{{#1}^{\rad{\phantom{.}}}}
\newcommand{\restr}[2]{{#1}_{|#2}}
\newcommand{\zclos}[1]{{#1}_z}
\newcommand{\zpart}[1]{{#1}^z}

\newcommand{\GenC}{\widetilde\C}
\newcommand{\GenCc}{\GenC_{\mathrm{cnt}}}
\newcommand{\GenK}{\widetilde\K}
\newcommand{\GenKc}{\GenK_{\mathrm{cnt}}}
\newcommand{\GenKs}{\GenK_{\mathrm{sm}}}
\newcommand{\GenR}{\widetilde\R}
\newcommand{\GenRc}{\GenR_{\mathrm{cnt}}}
\newcommand{\Mod}{{\mathcal M}}
\newcommand{\Null}{{\mathcal N}}

\newcommand{\Charsets}{\mathcal S}
\newcommand{\ClosedCharsets}{\mathcal A}

\newtheorem{thm}{Theorem}[section]
\newtheorem{lemma}[thm]{Lemma}
\newtheorem{prop}[thm]{Proposition}
\newtheorem{df}[thm]{Definition}
\newtheorem{cor}[thm]{Corollary}

\theoremstyle{remark}
\newtheorem{rem}[thm]{Remark}

\begin{document}
\title{Asymptotic ideals (ideals in the ring of Colombeau generalized constants with continuous parametrization)}
\author{A.~Khelif\footnote{Universit\'e Paris 7, Equipe de Logique}, D.~Scarpalezos\footnote{Universit\'e Paris 7, Centre de Math\'ematiques}, H.~Vernaeve\footnote{Ghent University, Department of Mathematics}}
\date{}
\maketitle

\begin{abstract}
We study the asymptotics at zero of continuous functions on $(0,1]$ by means of their asymptotic ideals, i.e., ideals in the ring of continuous functions on $(0,1]$ satisfying a polynomial growth condition at $0$ modulo rapidly decreasing functions at $0$. As our main result, we characterize maximal and prime ideals in terms of maximal and prime filters.
\end{abstract}

\section{Introduction}
In this paper, we study the asymptotic ideals of continuous functions $(0,1]\to\K$ (where $\K$ is one of the fields $\R$ or $\C$), i.e., ideals in the ring of continuous functions $\phi$ satisfying the following growth condition (usually called moderateness)
\[(\exists N\in\N) (\exists \eps_0>0) (\forall\eps\le \eps_0)\abs{\phi(\eps)}\le \eps^{-N}\]
modulo the ideal of continuous functions $\phi$ satisfying
\[(\forall n\in\N) (\exists \eps_0>0) (\forall\eps\le \eps_0)\abs{\phi(\eps)}\le \eps^{n}\]
(usually called negligibility). Apart from the obvious interest of such a study to asymptotic analysis, such equivalence classes of functions also naturally arise  in generalized function theory as the ring of generalized constants $\GenKc$ of the algebra of Colombeau generalized functions (see \S \ref{sec-prelim}).

The ring $\GenKc$ of generalized constants with continuous dependence on the parameter has been introduced and studied in \cite{BK}, where it is also shown that this ring is isomorphic to the ring of generalized constants with smooth dependence. In fact, the study of the ring $\GenKc$ amounts to the study of the asymptotics at zero of moderate continuous functions on $(0,1]$.

In generalized function theory, the choice of continuous dependence comes from the observation that when one embeds distributions in an algebra of Colombeau generalized functions and when one solves nonlinear problems, one always encounters generalized functions represented by continuous (even smooth) nets of smooth functions.

The algebraic properties of the ring $\GenKc$ are different from those of the ring $\GenK$ of generalized constants without continuous dependence on the parameter, and many tools used in the study of $\GenK$ cannot be used. Most strikingly, this is manifested by the fact that $\GenKc$ does not have any nontrivial idempotent elements, in sharp contrast with the ring $\GenK$ (which is a so-called exchange ring \cite{HVIdeals}). Thus the main tools used in \cite{AJ} and \cite{HVIdeals} to study $\GenK$ cannot be used.

In this paper, we study prime and maximal ideals by attaching a filter of closed subsets of $(0,1]$ to each ideal. The filter is analogous to the filter $\{S\subseteq (0,1]: e_{\co S}\in I\}$ attached to an ideal $I\idealproper\GenK$ (\cite[\S 6]{HVIdeals}), and thus allows us to overcome the difficulty of the lack of idempotents. In this way, we obtain a classification of maximal and minimal prime ideals in terms of maximal and prime filters.

The methods used in this paper are inspired by the study of the ideals in $\GenK$ \cite{AJ,HVIdeals} and by the study of maximal ideals of rings of continuous functions by Gillman and Jerison \cite{GJ}. Compared to \cite{GJ}, the main novelty is the adaptation to the asymptotic nature of the ring $\GenKc$.

\section{Preliminaries}\label{sec-prelim}
The ring $\GenK$, with $\K=\R$ or $\K=\C$ (the field of real, resp.\ complex numbers), is defined as $\Mod_\K/\Null_\K$, where
\begin{align*}
\Mod_\K&=\{(x_\eps)_\eps\in \K^{(0,1]}: (\exists N\in\N) (\exists \eps_0>0) (\forall\eps\le \eps_0)\abs{x_\eps}\le \eps^{-N}\}\\
\Null_\K&=\{(x_\eps)_\eps\in \K^{(0,1]}: (\forall n\in\N) (\exists \eps_0>0) (\forall\eps\le\eps_0) \abs{x_\eps}\le\eps^n\}.
\end{align*}
We denote by $[x_\eps]\in\GenK$ the element with representative $(x_\eps)_\eps$ and we denote $\caninf:= [\eps]$.

$\GenK$ is a complete topological ring with the so-called sharp topology, which can be defined as follows. Let $x=[x_\eps]\in\GenK$. Let
\[v(x):= \sup\{a\in\R: (\exists \eps_0>0)(\forall \eps\le\eps_0) \abs{x_\eps}\le \eps^a\}.\]
Then the ultrametric $d(x,y):= e^{-v(x-y)}$ induces a topology on $\GenK$ which is called the sharp topology \cite{Scarpa93}.

Denoting by $\Cnt((0,1])$ (resp.\ $\Cnt[\infty]((0,1])$ the set of continuous (resp.\ smooth) maps in $\K^{(0,1]}$, the ring $\GenKc := (\Mod_\K\cap \Cnt((0,1]))/(\Null_K\cap \Cnt((0,1]))$ and $\GenKs := (\Mod_\K\cap \Cnt[\infty]((0,1]))/(\Null_K\cap \Cnt[\infty]((0,1]))$. Clearly, $\GenKs\subseteq\GenKc\subseteq\GenK$. In \cite{BK}, it is shown that $\GenKc=\GenKs$.

We denote $I\idealproper\GenKc$ for a proper ideal $I$ of $\GenKc$ (i.e., $I\ne \GenKc$).

$\GenK$ is an exchange ring \cite{HVIdeals}, i.e., for each $a\in\GenK$, there exists an idempotent $e\in\GenK$ such that $a+e$ is invertible. Unlike $\GenK$, $\GenKc$ is not an exchange ring \cite[Lemma 4.3]{BK}.

Like $\GenK$, $\GenKc$ is a Gelfand ring \cite[Lemma 4.5]{BK}, i.e., every prime ideal is contained in a unique maximal ideal.

Like $\GenK$, $\GenKc$ is a Bezout ring \cite[Prop.\ 4.26]{BK}, i.e., every finitely generated ideal is principal.

Like $\GenR$, $\GenRc$ is an $l$-ring (or lattice-ordered ring) \cite[Prop.\ 4.13]{BK}.

Let $I\ideal\GenKc$ and $x\in I$. Then $\abs x\in I$ \cite[Lemma 4.24]{BK}.\\
Let $I\ideal\GenRc$. Then $I$ is an l-ideal (or absolutely (order) convex), i.e., if $x\in I$, $x'\in \GenRc$ and $\abs{x'}\le \abs{x}$, then $x'\in I$. \cite[Prop.\ 4.25]{BK}.

Let us point out explicitly the corollary that then also for $I\ideal\GenCc$, $z\in I$, $z'\in\GenCc$, $\abs{z'}\le \abs{z}$ implies that $z'\in I$. Indeed, $z\in I$ implies $\abs{z}\in I\cap\GenRc$ \cite[Lemma 4.24]{BK}. As $I\cap\GenRc\ideal\GenRc$, $I\cap\GenRc$ is an l-ideal in $\GenRc$. Hence $\abs{\Re z'}\le \abs{z}$ implies that $\Re z'\in I\cap\GenRc$. Similarly, $\Im z'\in I\cap\GenRc$. Thus $z' = \Re z' + i \Im z'\in I$.

Hence the bijective correspondence of ideals in $\GenKc$ takes the same form as for ideals in $\GenK$ (\cite{HVIdeals}): the map $I\ideal\GenCc\mapsto I\cap\GenRc = \{\Re z: z\in I\}\ideal\GenRc$ has as an inverse the map $J\ideal\GenRc \mapsto \langle J\rangle = \{z\in \GenCc: \abs z \in J\}\ideal\GenCc$ (where $\langle J\rangle$ is also the ideal generated by $J$ in $\GenCc$). It is an inclusion-preserving bijection between the lattice of ideals of $\GenCc$ and the lattice of ideals of $\GenRc$. In particular, arbitrary sums and intersections are preserved. One easily checks that the isomorphism also preserves products of ideals, principal, pseudoprime and irreducible ideals.
\smallskip

Let $R$ be a commutative ring with 1. An ideal $I\ideal R$ is pure if \cite[Prop.~7.2]{Borceux83}
\[(\forall x\in I)(\exists y\in I)(x=xy).\]
We denote by $\pure(I)$ the pure part of $I\ideal R$, i.e., the largest pure ideal contained in $I$ \cite[Prop.~7.8]{Borceux83}. By definition, $I$ is pure iff $I=\pure(I)$. If $R$ is a Gelfand ring, then \cite[\S 8.2--3]{Borceux83}
\[\pure(I)=\{x\in R: (\exists y\in I)(x=xy)\}.\]

An ideal $I\ideal R$ is idempotent if $I^2=I$.

We denote the radical of $I\ideal R$ by $\sqrt I=\{x\in R: (\exists n\in\N) x^n\in I\} = \bigcap_{I\subseteq P\atop P \text{ prime }} P$ (e.g., see \cite[0.18]{GJ}).

$I\ideal R$ is radical (or semiprime) if $I=\rad I$, or equivalently, if $(\forall x\in R)(x^2\in I\implies x\in I)$.\\
$I\ideal R$ is pseudoprime if for each $a,b\in R$, $ab=0$ implies $a\in I$ or $b\in I$.

$I\ideal R$ is irreducible (or meet-irreducible) if for each $J, K\ideal R$, $I=J\cap K$ implies $I=J$ or $I=K$ \cite[\S 6]{Matsumura}.

\section{Characteristic sets}
\begin{df}
A set $S\subseteq (0,1]$ such that $0\in \clos S$ (closure in $\R$) is called a \defstyle{characteristic set} \cite{BK}. We denote the set of all characteristic sets by $\Charsets$.

Let $S,T\in\Charsets$. We say that $T$ is an \defstyle{extension} of $S$ if $\clos S\subseteq \inter T$ (closure and interior in $(0,1]$) and denote this by $S\prec T$ (or equivalently, $T\succ S$). It is straightforward to check that $\prec$ is antireflexive and transitive on $\Charsets\setminus\{(0,1]\}$, and hence defines a partial order on $\Charsets\setminus\{(0,1]\}$. Notice that $(0,1]\prec (0,1]$, which will turn out to be convenient.
\end{df}

\begin{lemma}\label{ext-eltair}\leavevmode
Let $S,T\in\Charsets$.
\begin{enumerate}
\item If $S\prec T$, there exists $U\in\Charsets$ such that $S\prec U\prec T$.\\
In particular, $\prec$ is a dense order on $\Charsets\setminus\{(0,1]\}$.
\item $S\prec T$ iff $\co T\prec \co S$.
\end{enumerate}
\end{lemma}
\begin{proof}
1. Let $S\prec T$. By Urysohn's lemma, there exists $\phi\in\Cnt((0,1])$ such that $0\le\phi\le 1$, $\restr{\phi}{S}=0$ and $\restr{\phi}{\co T}= 1$. Let $U:=\{\eps\in(0,1]: \phi(\eps)\le 1/2\}$. Then $S\prec U\prec T$.

2. $\clos S\subseteq \inter T \iff \clos{(\co T)}=\co{(\inter T)}\subseteq \co{(\clos S)} =\inter{(\co S)}$.
\end{proof}

\begin{df}(cf.\ \cite[4.16]{BK})
Let $x\in\GenKc$ and $S\in\Charsets$. Then $\restr{x}{S}=0$ if
\[(\forall n\in\N) (\exists \delta>0) (\forall \eps\in S\cap (0,\delta)) (\abs{x_\eps}\le\eps^n).\]
where $(x_\eps)_\eps$ is any representative of $x$. We similarly write $\restr{x}{S}=\restr{y}{S}$ for $\restr{(x-y)}S=0$, $\restr x S =1$ for $\restr{(x-1)}S=0$, \dots\\
We say that $\restr{x}{S}$ is invertible if there exists $y\in\GenKc$ such that $\restr{(xy)}S=1$.
\end{df}

\begin{lemma}\label{inv-char}
Let $S\in\Charsets$.
\begin{enumerate}
\item Let $x\in\GenKc$. Then the following are equivalent:
\begin{enumerate}
\item $\restr{x}S$ is invertible (in $\GenKc$)
\item $\restr{x}S$ is invertible in $\GenK$
\item $\restr{x}S$ is bounded away from zero, i.e., for some representative $(x_\eps)_\eps$ of $x$,
\[(\exists n\in\N) (\exists \delta >0) (\forall \eps\in S\cap(0,\delta)) (\abs{x_\eps}\ge\eps^n).\]
(the statement then automatically holds for any representative $(x_\eps)_\eps$ of $x$).
\item for each characteristic set $T\subseteq S$, $\restr x T \ne 0$.
\end{enumerate}
\item $\{x\in\GenKc: \restr x S$ is invertible$\}$ is open.
\item $\restr x S=0$ iff for each characteristic set $T\subseteq S$, $\restr x T$ is not invertible.
\end{enumerate}
\end{lemma}
\begin{proof}
1. $(b)\Leftrightarrow (c)\Leftrightarrow (d)$: by \cite[Lemma 4.1]{HVIdeals}.

$(a)\Rightarrow(b)$: trivial.\\
$(c)\Rightarrow(a)$: let $T:= \{\eps\in (0,1]: \abs{x_\eps}>\eps^n/2\}$. As $(x_\eps)_\eps$ is continuous, $S\cap (0,\delta)\prec T$. By Urysohn's lemma, there exists $\phi\in\Cnt((0,1])$ such that $0\le\phi\le 1$, $\restr\phi {S\cap(0,\delta)}=1$ and $\restr\phi {\co T} = 0$. Let $y_\eps:= \phi(\eps)/x_\eps$, if $\eps\in T$ and $y_\eps:=0$, if $\eps\in\co T$. Then $\abs{y_\eps}\le 2\eps^{-n}$, $(y_\eps)_\eps\in$ is continuous and $x_\eps y_\eps = 1$ for each $\eps\in S\cap (0,\delta)$. Hence $(y_\eps)_\eps$ is a representative of some $y\in\GenKc$ with $\restr {(xy)} S = 1$.

2. Let $\restr x S$ be invertible. Let $n\in\N$ as in part 1(c). Then $\restr y S$ is invertible for each $y\in\GenKc$ with $\abs{x-y}\le \caninf^n/2$ (again by part 1(c)).

3. By \cite[Lemma 4.1]{HVIdeals}, since $\GenKc\subseteq \GenK$.
\end{proof}

\begin{prop}\label{zero-inv-ext}
Let $x\in\GenKc$ and $S\in \Charsets$.
\begin{enumerate}
\item If $\restr{x}S=0$, then $\restr x T=0$ for some $T\succ S$. 
\item If $\restr x S$ is invertible, then $\restr x T$ is invertible  for some $T\succ S$.
\end{enumerate}
\end{prop}
\begin{proof}
1. Let $(x_\eps)_{\eps\in(0,1]}$ be a (continuous) representative of $x$. Then for each $n\in\N$, there exist $\delta_n>0$ (w.l.o.g.\ strictly decreasing and tending to $0$) such that $\abs{x_\eps}\le \eps^n$ for each $\eps\in S$, $\eps\le \delta_n$. Then let 
$T := \bigcup_{n\in\N}(\delta_{n+2},\delta_n)\cap\{\eps\in(0,1]: \abs{x_\eps}\le 2\eps^n\}$. Then also $\restr x T=0$. We show that $S\prec T$. Let $\eps\in \clos S$. Then $\eps\in (\delta_{n+2},\delta_n)$ for some $n$. By continuity, also $\abs{x_\eps}\le\eps^n$ for each $\eps\in \clos S$, $\eps<\delta_n$. Hence $\eps$ belongs to the open set $(\delta_{n+2},\delta_n)\cap \{\eps\in(0,1]: \abs{x_\eps}<2 \eps^n\}\subseteq T$. Thus $\eps\in \inter T$.

2. Let $n\in\N$ and $\delta>0$ as in lemma \ref{inv-char}.1(c). Let $T:= \{\eps\in (0,1]: \abs{x_\eps}>\eps^n/2\}$ $\cup$ $(\delta/2,1)$. As $(x_\eps)_\eps$ is continuous, $S\prec T$. By lemma \ref{inv-char}.1(c), $\restr x T$ is invertible.
\end{proof}

\begin{lemma}\label{zero-product}
Let $a,b\in\GenKc$ and $S\in \Charsets$. If $\restr{(ab)}S=0$, then there exist closed $T,U$ with $S\subseteq \inter T\cup\inter U$ such that $\restr a T = 0$ and $\restr b U = 0$.
\end{lemma}
\begin{proof}
As $a,b\in\GenK$, there exists $V\subseteq S$ such that $\restr a V= 0$ and $\restr b {S\setminus V} = 0$ \cite{HVIdeals}. As $a,b\in\GenKc$, there exist (w.l.o.g.\ closed) $T$, $U$ with $V\prec T$, $S\setminus V\prec U$ such that $\restr a T = 0$ and $\restr b U = 0$ by Prop.\ \ref{zero-inv-ext}.
\end{proof}

\section{Asymptotic filters}
In \cite{GJ}, to any ideal $I\idealproper\Cnt(X)$ (with $X$ a topological space), a filter is associated consisting of the zero-sets of all $f\in I$ and conversely, to a filter $\Filter$ of zero-sets, an ideal $I$ is associated. Taking into account that there is no largest zero-set for $x\in \GenKc$, we proceed as follows:
\begin{df}
A filter of closed subsets of $(0,1]$ is a family $\Filter$ of (relatively) closed subsets of $(0,1]$ such that
\begin{enumerate}
\item $\emptyset\notin\Filter$
\item if $S,T\in\Filter$, then $S\cap T\in\Filter$
\item if $S\in\Filter$, $T\subseteq(0,1]$ is closed and $S\subseteq T$, then $T\in\Filter$.
\end{enumerate}

A closed characteristic subset of $(0,1]$ is called an \defstyle{asymptotic subset}. We denote the set of all asymptotic subsets by $\ClosedCharsets$.

An \defstyle{asymptotic filter} or \defstyle{a-filter} is a filter of closed subsets of $(0,1]$ that contains $(0,\delta]$ for each $\delta>0$. Notice that this implies that $\Filter\subseteq \ClosedCharsets$.

We define as follows a topology on $\ClosedCharsets$. Denoting open intervals corresponding to $\prec$ by
\[(S,T)_\prec:= \{U\in \ClosedCharsets: S\prec U \prec T\},\]
the \defstyle{extension topology} is the topology on $\ClosedCharsets$ with base $\{(S,T)_\prec: S, T\subseteq (0,1]\}$.
We will call $\prec$-open, $\prec$-closed, \dots sets that are open, closed, \dots for this topology. Notice that $\{(0,1]\}$ is $\prec$-open, which will turn out to be convenient.
\end{df}

\begin{rem}
A filter is called free (or non-principal) if $\bigcap_{S\in\Filter} S=\emptyset$. We can alternatively define an a-filter as a free filter of closed subsets of $(0,1]$. For, if $\Filter$ is a filter of closed subsets of $(0,1]$ and $(0,\delta]\notin\Filter$ for some $\delta>0$, then $S\cap [\delta,1]\ne\emptyset$ for each $S\in\Filter$. By compactness of $[\delta,1]$, it would then follow that $\bigcap_{S\in\Filter} S\cap [\delta,1]\ne\emptyset$.
\end{rem}

\begin{df}
Let $I\idealproper\GenKc$. Then
\[\Filter(I):= \{S\subseteq (0,1] \text{ closed}: (\exists x\in I) (\restr x {\co S} \text{ is invertible})\}\]
(here it is understood that $\restr x S$ is trivially invertible if $0\notin \clos S$).

Let $\Filter$ be an a-filter on $(0,1]$. Then
\[I(\Filter) := \{x\in\GenKc: (\exists S\in\Filter) (\restr x S=0)\}.\]
\end{df}

\begin{lemma}
For $I\idealproper\GenKc$,
\[\Filter(I) = \{S\subseteq (0,1] \text{ closed}: (\exists x\in I) (\restr x {\co S} = 1)\}.\]
\end{lemma}
\begin{proof}
If $x\in I$ and $\restr x{\co S}$ is invertible, then there exists $y\in\GenKc$ such that $\restr{(xy)}{\co S}=1$, and $xy\in I$.
\end{proof}

\begin{prop}\label{ideal-filter-eltary}
Let $I\idealproper \GenKc$ and $\Filter$ an a-filter on $(0,1]$.
\begin{enumerate}
\item $\Filter(I)$ is an a-filter on $(0,1]$.
\item $I(\Filter)\idealproper \GenKc$.
\item $\Filter(I(\Filter))\subseteq \Filter$.
\item $I(\Filter(I))\subseteq I$. 
\end{enumerate}
\end{prop}
\begin{proof}
1. Since a proper ideal does not contain invertible elements, $\emptyset\notin\Filter(I)$.\\
If $S,T\in\Filter(I)$, then there exist  $x,y\in I$ such that $\restr x {\co S}$ and $\restr y {\co T}$ are invertible. Hence also $\abs x^2 + \abs y^2 \in I$ and $\restr{(\abs x^2 + \abs y^2)}{\co S \cup \co T}$ is invertible, so also $S\cap T\in\Filter(I)$.\\
If $S\in\Filter(I)$, $T\subseteq (0,1]$ is closed and $S\subseteq T$, then clearly $T\in\Filter(I)$.\\
If $\delta>0$, then $0\notin\clos{\co{(0,\delta]}}$, hence $\restr x {\co{(0,\delta]}}$ is (trivially) invertible for each $x\in\GenKc$.

2. If $x,y\in I(\Filter)$, then $\restr x S = 0$ and $\restr y T = 0$ for some $S,T\in\Filter$. Then also $\restr{x+y}{S\cap T}=0$ and $S\cap T\in\Filter$, so $x+y\in I(\Filter)$. For $z\in\GenKc$, also $\restr{xz}{S} = 0$, so $xz\in I(\Filter)$. $1\notin I(\Filter)$, since $\restr 1 S\ne 0$ for each $S\in\Charsets$.

3. Let $S\in\Filter(I(\Filter))$. Then there exists $x\in I(\Filter)$ such that $\restr x {\co S} = 1$. So there exists $T\in\Filter$ such that $\restr x T = 0$. Then $T\cap (0,\delta]\subseteq S$ for some $\delta>0$. For otherwise, one constructs $V\subseteq T\cap \co S$ with $0\in \clos V$ such that $\restr x V = 0$, contradicting $\restr x V = 1$. Thus $S\in\Filter$.

4. Let $x\in I(\Filter(I))$. Then there exists $S\in \Filter(I)$ such that $\restr x S = 0$. So there exists $y\in I$ such that $\restr y {\co S} = 1$. As $x\in\GenKc$, $\abs{x}\le\caninf^{-N}$ for some $N\in\N$. Then $\abs{x}\le \caninf^{-N} \abs{y}$. As $\caninf^{-N} y\in I$ and ideals in $\GenKc$ are absolutely order convex \cite[Prop.\ 4.25]{BK}, $x\in I$.
\end{proof}

\begin{prop}
Let $\Filter$ be an a-filter on $(0,1]$. Then
\begin{enumerate}
\item $\inter\Filter = \{S\in\ClosedCharsets: (\exists T\prec S)(T\in\Filter)\}$ ($\inter\Filter$ denotes the $\prec$-interior).
\item $\inter\Filter$ is an a-filter.
\end{enumerate}
\end{prop}
\begin{proof}
1. $\subseteq$: let $\mathcal X\subseteq \Filter$ be $\prec$-open. If $S\in\mathcal X$, then $S\in (T,U)_\prec\subseteq \mathcal X$, for some $T,U\subseteq (0,1]$. W.l.og., $T$ is closed. We first show that there exists $V\prec S$ with $V\in\ClosedCharsets$.\\
Otherwise, $T \notin\Charsets$, i.e., $T\cap (0,\delta]=\emptyset$ for some $\delta>0$. As $S\in\Charsets$, we can construct $W_1, W_2\subseteq S$ with $W_1, W_2\in\ClosedCharsets$ and $W_1\cap W_2=\emptyset$. Then $W_j\cup[\delta/2,1]\in(T,U)_\prec\subseteq \Filter$. Hence also $\emptyset = W_1\cap W_2\cap (0,\delta/3]\in\Filter$, a contradiction.\\
Since $\prec$ is a dense order, $T\prec W\prec S$ for some closed $W$. Hence also $T\prec V\cup W\prec S$, and $V\cup W\in\ClosedCharsets$. Thus $V\cup W\in (T,U)_\prec\subseteq \mathcal X\subseteq \Filter$. Hence $\mathcal X\subseteq \{S\in\ClosedCharsets: (\exists T\prec S)(T\in\Filter)\}$.

$\supseteq$: $\{S\in\ClosedCharsets: (\exists T\prec S)(T\in\Filter)\}\subseteq \Filter$ and is $\prec$-open: if $T\prec S$ with $T\in\Filter$, then also $S\in(T,(0,1])_\prec\subseteq \{S\in\ClosedCharsets: (\exists T\prec S)(T\in\Filter)\}$.

2. As $\inter\Filter\subseteq\Filter$, $\emptyset\notin\inter\Filter$.\\
If $U\prec S$, $V\prec T$ with $U,V\in\Filter$, then also $U\cap V\prec S\cap T$ with $U\cap V\in\Filter$.\\
The other defining properties of an a-filter are immediately checked using part 1.
\end{proof}

\begin{thm}\label{ext-open-char}\leavevmode
\begin{enumerate}
\item For each a-filter $\Filter$ on $(0,1]$, $\Filter(I(\Filter)) = \inter \Filter$.
\item $\{\Filter(I): I\idealproper \GenKc\}$ is the set of $\prec$-open a-filters on $(0,1]$.
\end{enumerate}
\end{thm}
\begin{proof}
First, let $I\idealproper\GenKc$. We show that $\Filter(I)$ is $\prec$-open:\\
Let $S\in\Filter(I)$. Then there exists $x\in I$ such that $\restr x {\co S}$ is invertible. By proposition \ref{zero-inv-ext}, there exists $T\succ \co S$ such that $\restr x T$ is invertible. W.l.og.\ $T$ is open. Then $\co T\in \Filter(I)$ and $\co T\prec S$.\\
In particular, $\Filter(I(\Filter))\subseteq \Filter$ is $\prec$-open, and hence $\Filter(I(\Filter))\subseteq \inter\Filter$.\\
Conversely, we show that $\inter\Filter\subseteq \Filter(I(\Filter))$:\\
Let $S\in\inter\Filter$. Then there exists $T\prec S$ such that $T\in\Filter$. By Urysohn's lemma, there exists $x\in\GenKc$ such that $\restr x T = 0$ and $\restr x {\co S} = 1$. Hence $x\in I(\Filter)$ and $S\in\Filter(I(\Filter))$.\\
Finally, if an a-filter $\Filter$ is $\prec$-open, then $\Filter = \inter\Filter = \Filter(I(\Filter))$, hence $\Filter=\Filter(I$) for some $I\idealproper\GenKc$.
\end{proof}

\begin{thm}\label{pure-char}\leavevmode
\begin{enumerate}
\item For each $I\ideal\GenKc$, $I(\Filter(I)) = \pure(I)$.
\item $\{I(\Filter): \Filter$ is an a-filter on $(0,1]\}$ is the set of (proper) pure ideals in $\GenKc$.
\end{enumerate}
\end{thm}
\begin{proof}
First, let $\Filter$ be an a-filter on $(0,1]$. We show that $I(\Filter)$ is pure:\\
Let $x\in I(\Filter)$. Then there exists $S\in\mathcal F$ such that $\restr x S = 0$. By proposition \ref{zero-inv-ext}, $\restr x T = 0$ for some $T\succ S$. By Urysohn's lemma, there exists $y\in \GenKc$ such that $\restr y S=0$, $\restr y {\co T} = 1$. Then $\restr{(xy)}T = 0$ and $\restr{(xy)}{\co T} = \restr{x}{\co T}$. Hence $x=xy$ and $y\in I(\Filter)$.\\
In particular, $I(\Filter(I))\subseteq I$ is pure for each $I\idealproper \GenKc$, and hence $I(\Filter(I))\subseteq \pure(I)$.\\
Conversely, we show that $\pure(I)\subseteq I(\Filter(I))$ for each $I\idealproper \GenKc$:\\
Let $x\in \pure(I)$, i.e., there exists $y\in I$ such that $x=xy$. As $x(1-y)=0$, there exist (by lemma \ref{zero-product}) closed $S,T\subseteq (0,1]$ with $S\cup T=(0,1]$ such that $\restr x S = 0$ and $\restr{(1-y)} T = 0$. Hence $\restr{y}{\co S} =1$, so $S\in \Filter(I)$, and $x\in I(\Filter(I))$.\\
Finally, if $I\idealproper \GenKc$ is pure, then $I=\pure(I)= I(\Filter(I))$, hence $I=I(\Filter)$ for some a-filter $\Filter$ on $(0,1]$.
\end{proof}

\section{Closed ideals and filters}
We will denote $\clos I(\Filter):= \clos{I(\Filter)}$ (closure in the sharp topology) and $\clos\Filter(I):= \clos{\Filter(I)}$ ($\prec$-closure).

\begin{prop}
Let $\Filter$ be an a-filter on $(0,1]$. Then
\begin{enumerate}
\item $\clos\Filter = \{S\in\ClosedCharsets: (\forall T\succ S$, $T$ closed$)(T\in\Filter)\}$.
\item $\clos\Filter$ is an a-filter.
\end{enumerate}
\end{prop}
\begin{proof}
1. Call $\Filter^*:= \{S\in\ClosedCharsets: (\forall T\succ S$, $T$ closed$)(T\in\Filter)\}$.\\
$\subseteq$: $\Filter\subseteq \Filter^*$ and $\Filter^*$ is $\prec$-closed: if $S\in \ClosedCharsets\setminus \Filter^*$, then there exists a closed $T\succ S$ with $T\notin\Filter$, hence also $(\emptyset, T)_\prec\subseteq \ClosedCharsets\setminus \Filter^*$.

$\supseteq$: let $\mathcal X \supseteq\Filter$ be $\prec$-closed. Let $S\in\ClosedCharsets\setminus \mathcal X$. Then $S\in (T,U)_\prec\subseteq \ClosedCharsets\setminus \mathcal X$ for some $T,U\in\ClosedCharsets$. As $\prec$ is a dense order, $S\prec V\prec U$ for some closed $V$, and $V\in (T,U)_\prec\subseteq \ClosedCharsets\setminus\mathcal X\subseteq\ClosedCharsets\setminus \Filter$. Thus $S\notin\Filter^*$. Hence $\Filter^*\subseteq \mathcal X$.

2. $\emptyset\notin\clos\Filter$, since $\emptyset\notin\ClosedCharsets$.\\
Let $S_1$, $S_2$ $\in\clos\Filter$ and let $T\succ S_1\cap S_2$. Let
\begin{align*}
U_1 &= \{\eps\in(0,1]: d(\eps,S_1)<d(\eps,S_2)\}\\
U_2 &= \{\eps\in(0,1]: d(\eps,S_2)<d(\eps,S_1)\}.
\end{align*}
Let $V_1:= U_1\cup \inter T$ and $V_2:= U_2\cup \inter T$. Then $S_1 = (S_1\setminus S_2) \cup (S_1\cap S_2)\subseteq U_1 \cup \inter T = V_1$ since $S_2$ is closed. Since $V_1$ is open, $S_1\prec V_1$. Hence $V_1\in\Filter$. Similarly, $V_2\in\Filter$. As $V_1\cap V_2\subseteq T$, $T\in\Filter$. We conclude that $S_1\cap S_2\in\clos\Filter$.\\
The other defining properties of an a-filter are immediately checked using part 1.
\end{proof}

\begin{cor}\label{clos-of-open}
If $\Filter$ is an a-filter on $(0,1]$, then $\clos{\inter\Filter}=\clos\Filter$.
\end{cor}
\begin{proof}
$\supseteq$: since $\inter\Filter\subseteq\Filter$.\\
$\subseteq$: it suffices to show that $\Filter\subseteq\clos{\inter\Filter}$. Let $S\in\Filter$. Let $T\subseteq(0,1]$ be closed such that $T\succ S$. Then $T\in \inter\Filter$. Hence $S\in\clos{\inter\Filter}$.
\end{proof}

\begin{thm}\label{clos-ideal-char}
Let $\Filter$ be an a-filter. Then
\[\clos I(\Filter)= \{x\in\GenKc: (\forall S\in\ClosedCharsets) (\restr{x}{\co S} \text{ invertible} \implies S\in \Filter)\}.\]
\end{thm}
\begin{proof}
Call $I^+(\Filter):=\{x\in\GenKc: (\forall S\in\ClosedCharsets) (\restr{x}{\co S} \text{ invertible} \implies S\in \Filter)\}$.\\
We first show that $I^+(\Filter)$ is closed:\\
If $a\in\GenKc\setminus I^+(\Filter)$, then there exists $S\in\ClosedCharsets\setminus \Filter$ such that $\restr a {\co S}$ is invertible. By lemma \ref{inv-char}, $\restr x {\co S}$ is invertible for each $x$ in a certain neighborhood of $a$. Then such $x\notin I^+(\Filter)$, too. Hence $\GenKc\setminus I^+(\Filter)$ is open.\\
We now show that $I(\Filter)\subseteq I^+(\Filter)$:\\
Let $x\in I(\Filter)$. Then $\restr x S=0$ for some $S\in\Filter$. Let $T\in\ClosedCharsets$ such that $\restr x {\co T}$ is invertible. Then $S\cap(0,\delta)\setminus T=\emptyset$ for some $\delta >0$, for otherwise, $0\in \clos{S\setminus T}$ and $\restr x {S\setminus T} = 0$ and $\restr x {S\setminus T}$ is invertible, a contradiction. Hence $S\cap (0,\delta)\subseteq T$, and $T\in\Filter$. Thus $x\in I^+(\Filter)$.\\
Finally, we show that $I^+(\Filter)\subseteq \clos I(\Filter)$:\\
Let $x=[x_\eps]\in I^+(\Filter)$. Consider the sets $L_n:= \{\eps: \abs{x_\eps}> \eps^n\}$. As $\restr x {L_n}$ is invertible, $\co L_n\in\Filter$, for each $n\in\N$. Further, $L_n\prec L_{n+1}$ for each $n\in\N$. By Urysohn's lemma, there exist $y_n\in\GenKc$ such that $\restr{y_n}{L_n}=1$ and $\restr{y_n}{\co L_{n+1}} = 0$ and $0\le y_n\le 1$. Then $\restr{\abs{x y_{n} - x}}{L_n}=0$ and $\restr{\abs{x y_{n} - x}}{\co L_n}\le \restr{\abs x}{\co L_n}\le \caninf^n$. Hence $\abs{x y_n - x}\le \caninf^n$, and $\lim_{n\to\infty} xy_n = x$. As $\restr{(x y_n)}{\co L_{n+1}} = 0$, $xy_n\in I(\Filter)$, for each $n$.
\end{proof}

\begin{cor}\label{clos-of-pure}
Let $I\idealproper\GenKc$. Then $I(\Filter(I))\subseteq I\subseteq \clos I(\Filter(I))$ and $\clos I = \clos{\pure(I)}$. 
\end{cor}
\begin{proof}
$I\subseteq \clos I(\Filter(I))$: let $x\in I$. Let $S\in\ClosedCharsets$ such that $\restr x {\co S}$ is invertible. Then $S\in\Filter(I)$. Hence by theorem \ref{clos-ideal-char}, $x\in \clos I(\Filter(I))$.\\
By proposition \ref{ideal-filter-eltary}, $I(\Filter(I))\subseteq I$. Hence $\clos I = \clos I(\Filter(I)) = \clos{\pure(I)}$ by theorem \ref{pure-char}.
\end{proof}

\begin{thm}
Let $I\idealproper\GenKc$. Then
\[\clos \Filter(I)=\{S\in\ClosedCharsets:(\forall x\in \GenKc) (\restr x S =0\implies x\in I)\}.\]
\end{thm}
\begin{proof}
Call $\Filter^+(I):=\{S\in\ClosedCharsets:(\forall x\in \GenKc) (\restr x S =0\implies x\in I)\}$.\\
We show that $\Filter^+(I)$ is closed:\\
Let $S\in \clos{\Filter^+}(I)$, i.e. $S\in\ClosedCharsets$ and $T\in \Filter^+(I)$, for each closed $T\succ S$. Let $x\in\GenKc$ such that $\restr x S = 0$. By lemma \ref{zero-inv-ext}, there exists $T\succ S$ such that $\restr x T = 0$. W.l.o.g, $T$ is closed. Thus $x\in I$. Hence $S\in \Filter^+(I)$.\\
We now show that $\Filter(I)\subseteq \Filter^+(I)$:\\
Let $S\in\Filter(I)$. Then there exists $a\in I$ such that $\restr a {\co S} = 1$. Now let $x\in \GenKc$ such that $\restr x S = 0$. Then $\restr x T = 0$ for some $T\succ S$. By Urysohn's lemma, there exists $y\in \GenKc$ with $\restr y S = 0$ and $\restr y {\co T} = 1$. Then $\restr{(xya)} T = \restr x T = 0$ and $\restr{(xya)}{\co T} = \restr x{\co T}$. Hence $x=xya\in I$.\\
Finally, we show that $\Filter^+(I)\subseteq \clos \Filter(I)$:\\
Let $S\in\Filter^+(I)$ and let $T\succ S$ be closed. By Urysohn's lemma, there exists $y\in\GenKc$ such that $\restr y S=0$ and $\restr y {\co T} = 1$. As $S\in\Filter^+(I)$, $y\in I$. Hence $T\in \Filter(I)$.
\end{proof}

\begin{thm}\label{equal-filter-of-ideal}
If $I,J\idealproper\GenKc$, then $\Filter(I)=\Filter(J)\iff \pure(I) = \pure(J) \iff \clos I = \clos J$.
\end{thm}
\begin{proof}
1. If $\Filter(I) = \Filter(J)$, then $\pure(I) = I(\Filter(I)) = I(\Filter(J)) = \pure(J)$ by theorem \ref{pure-char}.\\
2. If $\pure(I)=\pure(J)$, then $\clos I = \clos{\pure(I)} = \clos{\pure(J)} = \clos J$ by corollary \ref{clos-of-pure}.\\
3. Let $S\in \Filter(\clos I)$. Let $E:=\{x\in\GenKc: \restr x {\co S}$ is invertible$\}$. Then $\clos I\cap E\ne\emptyset$. By lemma \ref{inv-char}, $E$ is open, hence also $I\cap E\ne\emptyset$, i.e., $S\in \Filter(I)$.\\
Hence, if $\clos I=\clos J$, then $\Filter(I) = \Filter(\clos I) = \Filter(\clos J) = \Filter(J)$.
\end{proof}
\begin{cor}
If $I\idealproper\GenKc$, then $\pure(\clos I)=\pure(I)$.
\end{cor}

\begin{thm}\label{equal-ideal-of-filter}
Let $\Filter_1,\Filter_2$ be a-filters on $(0,1]$. Then $I(\Filter_1)= I(\Filter_2)\iff \inter\Filter_1 = \inter\Filter_2 \iff \clos {\Filter_1} = \clos {\Filter_2}$.
\end{thm}
\begin{proof}
1. If $I(\Filter_1)=I(\Filter_2)$, then $\inter\Filter_1 = \Filter(I(\Filter_1)) = \Filter(I(\Filter_2)) = \inter\Filter_2$ by theorem \ref{ext-open-char}.\\
2. If $\inter\Filter_1 = \inter\Filter_2$, then $\clos{\Filter_1} = \clos{\inter\Filter_1} = \clos{\inter\Filter_2} = \clos{\Filter_2}$ by corollary \ref{clos-of-open}.\\
3. Let $x\in I(\clos\Filter)$. Then $\restr x S = 0$ for some $S\in\clos\Filter$. By proposition \ref{zero-inv-ext}, there exists $T\succ S$ (w.l.o.g.\ $T$ closed) such that $\restr x T = 0$. So $T\in\Filter$, and $x\in I(\Filter)$.\\
Hence, if $\clos{\Filter_1} = \clos{\Filter_2}$, then $I(\Filter_1) = I(\clos{\Filter_1})= I(\clos{\Filter_2}) = I(\Filter_2)$.
\end{proof}
\begin{cor}
If $\Filter$ is an a-filter on $(0,1]$, then $\inter{(\clos\Filter)}=\inter\Filter$.
\end{cor}

\section{Maximal and prime ideals and filters}
\begin{df}
An a-filter $\Filter$ on $(0,1]$ is called prime if for each $S, T\in\ClosedCharsets$ with $S\cup T\in\Filter$, either $S\in\Filter$ or $T\in\Filter$.\\
An a-filter $\Filter$ on $(0,1]$ is called pseudoprime if for each $S, T\in\ClosedCharsets$ with $\inter S\cup \inter T= (0,1]$, either $S\in\Filter$ or $T\in\Filter$.
\end{df}

\begin{rem}
1. In the definition of (pseudo)prime a-filter, we may also ask the condition for each closed $S,T\subseteq (0,1]$ (instead of for each $S,T\in\ClosedCharsets$ only). For, if $S\notin\ClosedCharsets$, then $S\notin\Charsets$, i.e., $(0,\delta]\cap S=\emptyset$ for some $\delta>0$. Hence $(S\cup T)\cap (0,\delta]\subseteq T$. So if $S\cap T\in\Filter$, then also $T\in\Filter$. The case $T\notin\ClosedCharsets$ is symmetric.

2. An a-filter $\Filter$ on $(0,1]$ is prime if and only if for each $S, T\in\ClosedCharsets$ with $S\cup T= (0,1]$, either $S\in\Filter$ or $T\in\Filter$. For, if $\Filter$ satisfies the latter condition  and $S\cup T\in\Filter$, we consider
\[U:= \{\eps\in(0,1]: d(\eps,S)\le d(\eps,T)\}\quad\text{and}\quad
V:= \{\eps\in(0,1]: d(\eps,T)\le d(\eps,S)\}.\]
Then $U$, $V$ are closed with $U\cup V=(0,1]$. Hence $U\in \Filter$ or $V\in\Filter$. If $U\in\Filter$, then also $(S\cup T)\cap U\in\Filter$. As $(S\cup T)\cap U\subseteq S$, also $S\in\Filter$. The case $V\in\Filter$ is symmetric.

This motivates our (less obvious) definition of pseudoprime a-filter.
\end{rem}

\begin{lemma}\label{prec-union}
Let $S,T,U\subseteq (0,1]$ be open and nonempty with $\clos U\subseteq S\cup T$. Then there exist $V\prec S$ and $W\prec T$ such that $U\subseteq V\cup W$.
\end{lemma}
\begin{proof}
Let
\begin{align*}
V&:= \{\eps\in(0,1]: \max(d(\eps,U), d(\eps,\co T))\le d(\eps, \co S)\}\\
W&:= \{\eps\in(0,1]: \max(d(\eps,U), d(\eps,\co S))\le d(\eps, \co T)\}.
\end{align*}
If $\eps\in V\setminus S$, then $\eps\in \co T\cap \clos U\subseteq \co T\cap (S\cup T)\subseteq S$. Hence $V\subseteq S$. As $V$ is closed and $S$ is open, also $V\prec S$. Similarly $W\prec T$.\\
Further, let $\eps\in U$. Then either $d(\eps,\co T)\le d(\eps,\co S)$ (hence $\eps\in V$) or $d(\eps,\co S)\le d(\eps,\co T)$ (hence $\eps\in W$). So $U\subseteq V\cup W$.
\end{proof}

\begin{lemma}\label{pseudoprime-filter-ideal}
Let $\Filter$ be a pseudoprime a-filter on $(0,1]$. Then $I(\Filter)$ is pseudoprime.
\end{lemma}
\begin{proof}
Let $xy=0$. By lemma \ref{zero-product}, there exist closed $T,U$ with $\inter T\cup\inter U = (0,1]$ such that $\restr x T = 0$ and $\restr y U = 0$. As $\Filter$ is pseudoprime, $T\in\Filter$ or $U\in\Filter$. Hence $x\in I(\Filter)$ or $y\in I(\Filter)$.
\end{proof}

\begin{lemma}\label{pseudoprime-ideal-filter}
Let $I\idealproper\GenKc$ be pseudoprime. Then $\Filter(I)$ is pseudoprime.
\end{lemma}
\begin{proof}
Let $S, T\in\ClosedCharsets$ with $\inter S\cup\inter T = (0,1]$. Let $V\prec\inter S$ and $W\prec\inter T$ such that $V\cup W=(0,1]$ (lemma \ref{prec-union} with $U=(0,1]$). By Urysohn's lemma, there exist $x,y\in \GenKc$ such that $\restr x V=0$, $\restr x {\co S} = 1$, $\restr y W = 0$ and $\restr y {\co T} = 1$. Then $xy=0$. As $I$ is pseudoprime, $x\in I$ or $y\in I$. Hence $S\in \Filter(I)$ or $T\in \Filter(I)$.
\end{proof}

\begin{lemma}\label{closed-is-radical}
Every closed ideal $I\idealproper \GenKc$ is radical.
\end{lemma}
\begin{proof}
Let $S\in\Filter(\rad I)$. Then there exists $x\in\GenKc$ and $n\in\N$ with $x^n\in I$ and $\restr x {\co S}= 1$. Then also $\restr {x^n} {\co S}= 1$, hence $S\in\Filter(I)$. Thus $\Filter(\rad I)=\Filter(I)$. By theorem \ref{equal-filter-of-ideal}, $I \subseteq \rad I \subseteq \clos{\rad I} =\clos I = I$.
\end{proof}

\begin{prop}\label{pseudoprime-ideal-char}
Let $I\idealproper\GenKc$. Then the following are equivalent:
\begin{enumerate}
\item $I$ is pseudoprime
\item the set of ideals containing $I$ is totally ordered (for $\subseteq$)
\item $I$ is irreducible
\item $\rad I$ is prime
\item $\Filter(I)$ is pseudoprime.
\end{enumerate}
For $\GenKc=\GenRc$, this is still equivalent with
\begin{enumerate}
\item[6.] $\GenRc/I$ is totally ordered.
\end{enumerate}
\end{prop}
\begin{proof}
$1\implies 6$ (for $\GenKc=\GenRc$): let $a\in\GenRc$. Since $a^2=\abs{a}^2$, we have $(a-\abs a)(a+ \abs a)=0$. As $I$ is pseudoprime, $a-\abs a\in I$ or $a+\abs a\in I$. As $\GenRc$ is an $l$-ring, it follows that $a+I\ge 0$ or $-a+I\ge 0$ in $\GenRc/I$ (cf.~\cite[Thm.\ 5.3]{GJ}).\\
$6\implies 2$ (for $\GenKc=\GenRc$, cf.\ \cite[4.1]{GK60}): the map $J\mapsto J/I$ is an order preserving bijection between the ($l$-)ideals of $\GenRc$ containing $I$ and the $l$-ideals of $\GenRc/I$. As in any totally ordered ring, the $l$-ideals in $\GenRc/I$ are totally ordered.\\
$1\implies 2$ (for $\GenKc=\GenCc$): by the bijective correspondence of ideals in $\GenRc$ and in $\GenCc$ (section \ref{sec-prelim}).\\
$2\implies 3$: let $K=I\cap J$. Either $I\subseteq J$ or $J\subseteq I$, whence $K= I$ or $K=J$.\\
$3\implies 1$: as in any commutative $l$-ring with $1$ in which every ideal is an $l$-ideal, the irreducibility of $I\ideal\GenRc$ is equivalent with: for any $x,y\in\GenRc$, $x\GenRc\cap y\GenRc\subseteq I$ implies $x\in I$ or $y\in I$ \cite[Prop.\ 8.4.1]{BKW}.
So let $x,y\in\GenRc$ with $xy=0$. By lemma \ref{zero-product}, there exist open $T,U$ with $T\cup U = (0,1]$ such that $\restr x T = 0$ and $\restr y U = 0$. Let $z\in x\GenRc\cap y\GenRc$. Then $\restr z T= \restr z U = 0$, hence $z=0$. In particular, $x\GenRc\cap y\GenRc\subseteq I$, and hence $x\in I$ or $y\in I$. The bijective correspondence of ideals in $\GenRc$ and $\GenCc$ yields the result for $\GenCc$.\\
$2\implies 4$: the intersection of a chain of prime ideals is prime, hence $\rad I = \bigcap_{I\subseteq P, P \text{ prime}} P$ is prime.\\
$4\implies 5$: by lemma \ref{pseudoprime-ideal-filter}, $\Filter(\rad I)$ is pseudoprime. By the proof of lemma \ref{closed-is-radical}, $\Filter(I)=\Filter(\rad I)$.\\
$5\implies 1$: by lemma \ref{pseudoprime-filter-ideal}, $\pure(I)= I(\Filter(I))$ is pseudoprime. Hence $I\supseteq I(\Filter(I))$ is also pseudoprime.
\end{proof}

\begin{thm}\label{prime-ideal-char}
Let $I\idealproper\GenKc$. Then $I$ is prime iff $I$ is pseudoprime and radical.
\end{thm}
\begin{proof}
$\Rightarrow$: as $I$ is prime, $\rad I = \bigcap_{I\subseteq P,P \text{ prime}} P = I$.\\
$\Leftarrow$: $I=\rad I$ is prime by proposition \ref{pseudoprime-ideal-char}.
\end{proof}

\begin{lemma}\label{pure-is-radical}
Every pure ideal $I\idealproper\GenKc$ is radical.
\end{lemma}
\begin{proof}
Let $x^n\in I$ for some $x\in\GenKc$ and $n\in\N$. As $I=\pure(I)=I(\Filter(I))$, there exists $S\in\Filter(I)$ such that $\restr {x^n} S = 0$. Hence also $\restr{x} S = 0$, and $x\in I(\Filter(I)) = I$.
\end{proof}

\begin{prop}
For $I\idealproper \GenKc$, the following are equivalent:
\begin{enumerate}
\item $I$ is pseudoprime
\item $\pure(I)$ is prime
\item $I$ contains a prime ideal.
\end{enumerate}
\end{prop}
\begin{proof}
$1\implies 2$: by lemmas \ref{pseudoprime-filter-ideal} and \ref{pseudoprime-ideal-filter}, $\pure(I)=I(\Filter(I))$ is pseudoprime. By lemma \ref{pure-is-radical}, $\pure(I)$ is radical. Hence $\pure(I)$ is prime.\\
$2\implies 3$: $\pure(I)\subseteq I$.\\
$3\implies 1$: if $P\subseteq I$ is prime and $xy=0$, then $xy\in P$, so $x\in P\subseteq I$ or $y\in P\subseteq I$.
\end{proof}

\begin{prop}\label{pseudoprime-filter-kar}
Let $\Filter$ be an a-filter on $(0,1]$. Then the following are equivalent:
\begin{enumerate}
\item $\Filter$ is pseudoprime
\item $I(\Filter)$ is pseudoprime
\item $I(\Filter)$ is prime.
\end{enumerate}
\end{prop}
\begin{proof}
$1\implies 2$: by lemma \ref{pseudoprime-filter-ideal}, $I(\Filter)$ is pseudoprime.\\
$2\implies 3$: as $I(\Filter)$ is pure, $I(\Filter)$ is radical (lemma \ref{pure-is-radical}). By theorem \ref{prime-ideal-char}, $I(\Filter)$ is prime.\\
$3\implies 1$: by lemma \ref{pseudoprime-ideal-filter}, $\Filter(I(\Filter))$ is pseudoprime. As $\Filter(I(\Filter))\subseteq \Filter$, also $\Filter$ is pseudoprime.
\end{proof}

We now consider maximal ideals and a-filters:
\begin{thm}\label{max-filter}
Let $\Filter$ be an a-filter.
\begin{enumerate}
\item if $\Filter$ is pseudoprime, then $\clos\Filter$ is maximal.
\item $\Filter$ is maximal if and only if $\Filter$ is prime and $\prec$-closed.
\end{enumerate}
\end{thm}
\begin{proof}
1. Suppose $\clos\Filter\subsetneq \Filter'$ for some a-filter $\Filter'$. Let $S\in\Filter'\setminus \clos\Filter$. Then there exists a closed $T\succ S$ such that $T\notin \Filter$. As $\prec$ is a dense order, there exists an open $V$ with $S\prec V\prec T$. Since $\inter T\cup \inter{(\co V)} = (0,1]$ and $\Filter$ is pseudoprime, $\co V\in\Filter$. But then $\emptyset = S\cap \co V\in\Filter'$, a contradiction.

2. $\Rightarrow$: we show that $\Filter$ is closed: as $\Filter\subseteq\clos\Filter$, and $\clos\Filter$ is an a-filter, $\Filter=\clos\Filter$ by maximality. Further, we show that $\Filter$ is prime: let $S,T\in\ClosedCharsets$ such that $S\cup T\in\Filter$. Suppose there exists $U\in\Filter$ such that $U\cap S = \emptyset$ and there exists $V\in\Filter$ such that $V\cap T=\emptyset$. Then $\emptyset = (U\cap V)\cap (S\cup T) \in\Filter$, a contradiction. We may thus assume that $U\cap S\ne\emptyset$, for each $U\in\Filter$. (The case $U\cap T\ne\emptyset$, for each $U\in\Filter$ is similar.) Then $\emptyset\notin\Filter':=\{U\subseteq(0,1]$ closed$: (\exists V\in\Filter) (S\cap V\subseteq U)\}$. As $\Filter'$ is an a-filter, $\Filter=\Filter'$ by maximality. Hence $S\in\Filter$.

2. $\Leftarrow$: by part 1, $\Filter = \clos\Filter$ is maximal.
\end{proof}

\begin{thm}\label{max-ideal}
Let $I\idealproper\GenKc$.
\begin{enumerate}
\item if $I$ is pseudoprime, then $\clos I$ is maximal.
\item $I$ is maximal if and only if $I$ is prime and closed.
\end{enumerate}
\end{thm}
\begin{proof}
1. By proposition \ref{pseudoprime-ideal-char}, $\Filter(I)=\Filter(\clos I)$ is pseudoprime. Thus by theorem \ref{max-filter}, $\clos\Filter(\clos I)$ is maximal. Now let $\clos I\subseteq J\idealproper\GenKc$. Then $\clos\Filter(\clos I)\subseteq \clos\Filter(J)$, and hence $\clos\Filter(\clos I) = \clos\Filter(J)$ by maximality. Hence also $\pure(\clos I) = I(\Filter(\clos I)) = I(\clos\Filter(\clos I)) = I(\clos\Filter(J)) = I(\Filter(J)) = \pure(J)$, and hence $J \subseteq \clos J = \clos I$ by theorem \ref{equal-filter-of-ideal}.

2. $\Rightarrow$: let $E$ denote the set of invertible elements in $\GenKc$. As $I$ is a proper ideal, $I\cap E=\emptyset$. As $E$ is open, also $\clos I\cap E=\emptyset$. Hence $\clos I$ is proper, and $I=\clos I$ by maximality. Maximal ideals are prime in any commutative ring with $1$.\\
$\Leftarrow$: by part 1, $I=\clos I$ is maximal.
\end{proof}
\begin{cor}\label{char-min-max-prime-ideals}\leavevmode\par
1. The set of minimal prime ideals in $\GenKc$ equals
\[\{I(\Filter): \Filter \text{ is a max.\ a-filter on }(0,1]\}=\{I(\Filter): \Filter \text{ is a pseudoprime a-filter on }(0,1]\}.\]
2. The set of maximal ideals in $\GenKc$ equals
\[\{\clos I(\Filter): \Filter \text{ is a max.\ a-filter on }(0,1]\}=\{\clos I(\Filter): \Filter \text{ is a pseudoprime a-filter on }(0,1]\}.\]
\end{cor}
\begin{proof}
1.(a) Let $I\idealproper\GenKc$ be a minimal prime. Then $\Filter(I)$ is pseudoprime, and $I(\Filter(I))\subseteq I$ is a prime ideal. By minimality, $I=I(\Filter(I)) = I(\clos \Filter(I))$ and $\clos\Filter(I)$ is maximal.

(b) Let $\Filter$ be a pseudoprime a-filter on $(0,1]$. Then $I(\Filter)$ is prime by proposition \ref{pseudoprime-filter-kar}. If $P\idealproper\GenKc$ is prime with $P\subseteq I(\Filter)$, then $\Filter(P)\subseteq \Filter(I(\Filter))\subseteq \Filter$, and hence $\clos\Filter(P)\subseteq \clos\Filter$. As $P$ is prime, $\Filter(P)$ is pseudoprime, and hence $\clos\Filter(P)$ is maximal by theorem \ref{max-filter}. Hence $\clos\Filter(P)=\clos\Filter$. Consequently, $P\supseteq I(\Filter(P)) = I(\clos\Filter(P)) = I(\clos\Filter)= I(\Filter)$.

2.(a) Let $I\idealproper\GenKc$ be maximal. Then $I$ is pseudoprime, hence $\Filter(I)$ is pseudoprime, and thus $\clos\Filter(I)$ is maximal. Further, $I=\clos I = \clos{\pure(I)} = \clos I(\Filter(I)) = \clos I(\clos \Filter(I))$.

(b) Let $\Filter$ be a pseudoprime a-filter on $(0,1]$. Then $I(\Filter)$ is pseudoprime, hence $\clos I(\Filter)$ is maximal.
\end{proof}

\begin{prop}\label{intersection-of-max-ideals}
Let $I\idealproper \GenK$. Then $\clos I =\bigcap_{I\subseteq M\atop {M\text{ maximal}}}M$.\\
In particular, an ideal $I\idealproper\GenK$ is closed iff it is an intersection of maximal ideals.
\end{prop}
\begin{proof}
$\subseteq$: by theorem \ref{max-ideal}, maximal ideals are closed.\\
$\supseteq$: let $x\notin \clos I = \clos{\pure(I)} = \clos I(\Filter(I))$ (corollary \ref{clos-of-pure}). By theorem \ref{clos-ideal-char}, there exists $S\in\ClosedCharsets\setminus \Filter(I)$ such that $\restr x {\co S}$ is invertible. Let $E:=\{x\in\GenKc: \restr x {\co S}$ is invertible$\}$. As $E$ is closed under multiplication and $E\cap I=\emptyset$, there exists a prime $P\idealproper\GenKc$ such that $I\subseteq P$ and $E\cap P=\emptyset$ (e.g., \cite[0.16]{GJ}). As $E$ is open (lemma \ref{inv-char}), also $E\cap \clos P = \emptyset$. In particular, $\clos P$ is maximal and $x\notin \clos P$.
\end{proof}

\begin{rem}
In the previous, we showed that maximal ideals of $\GenKc$ are in bijective correpondence with maximal a-filters, which are in bijective correspondence with points of $\beta(0,1]\setminus(0,1]$, where $\beta(0,1]$ denotes the Stone-\v Cech compactification of $(0,1]$ (cf.\ \cite[6.5]{GJ}).
\end{rem}

\section{Rapid a-filters}
\begin{df}
An a-filter $\Filter$ is called \defstyle{rapid} if for each sequence $(S_n)_n$ in $\Filter$ with $S_1\succ S_2\succ \dots$, there exists $T\in\Filter$ such that $T\setminus S_n\notin\Charsets$.
\end{df}

\begin{thm}\label{rapid-filter-closed-ideal}
Let $\Filter$ be an a-filter. Then $I(\Filter)$ is closed iff $\Filter$ is rapid.
\end{thm}
\begin{proof}
$\Leftarrow$: let $a\in \clos I(\Filter)$ with continuous representative $(a_\eps)_\eps$. For each $n\in\N$, let $S_n:=\{\eps\in (0,1]: \abs{a_\eps}\le \eps^n\}$. By theorem \ref{clos-ideal-char}, $S_n\in\Filter$, and also $S_1\succ S_2\succ \dots$. As $\Filter$ is rapid, there exists $T\in\Filter$ such that $T\setminus S_n\notin\Charsets$. Hence $\restr{\abs{a}}{T}\le \caninf^n$, for each $n\in\N$, i.e., $\restr{a}T=0$. Hence $a\in I(\Filter)$.

$\Rightarrow$: let $S_n\in\Filter$, and also $S_1\succ S_2\succ \dots$. By Urysohn's lemma, there exist $\phi_n\in\Cnt((0,1])$ such that $0\le\phi_n\le \eps^n$, $\restr{\phi_n}{S_{n+1}}=0$ and $\restr{\phi_n}{\co S_n}= \eps^n$. Let $\phi:=\sum_{n=1}^\infty \phi_n$ on $(0, 1/2]$. By uniform convergence, $\phi$ is continuous and $\eps^{n+1}\le \phi(\eps)\le \eps^n+\eps^{n+1} +\dots \le 2\eps^n$ on $(0,1/2]\cap S_n\setminus S_{n+1}$. Extend $\phi$ to a continuous map on $(0,1]$. Then $a:=[\phi(\eps)]\in\GenKc$.\\
Let $T\in\ClosedCharsets$ be such that $\restr a{\co T}$ is invertible. Then there exists $n\in\N$ such that $\abs{\phi(\eps)}> 2\eps^n$ for $\eps\in\co T\cap (0,\delta]$ (some $0<\delta\le 1/2$). Hence $S_n\cap (0,\delta]\subseteq T$, and $T\in\Filter$. By theorem \ref{clos-ideal-char}, $a\in \clos I(\Filter)=I(\Filter)$. Thus there exists $T\in\Filter$ such that $\restr aT=0$.\\
Let $n\in\N$. Then $\abs{\phi(\eps)}<\eps^n$ for each $\eps\in(0,\delta]\cap T$ (some $0<\delta\le 1/2$). Hence $(0,\delta]\cap T\setminus S_n = \emptyset$.
\end{proof}

\begin{rem}
Recall that a filter $\Filter$ of subsets of $\N$ is called rapid if for any decreasing sequence $(S_n)_n$ in $\Filter$, there exists $S\in\Filter$ such that $S\setminus S_n$ is finite for every $n\in\N$. A free ultrafilter $\mathcal U$ of subsets of $\N$ is called weakly selective (or $\delta$-stable or P-point of $\beta\N\setminus \N$) if for each sequence $(S_n)_n$ in $\mathcal U$, there exists $S\in\mathcal U$ such that $S\setminus S_n$ is finite for each $n\in\N$. There exist weakly selective free ultrafilters if we assume the continuum hypothesis \cite{Rudin,Choquet68} (in fact, it satisfies to assume weaker axioms, e.g.\ ZFC+Martin's axiom \cite[\S 4]{Booth}). By definition, a weakly selective free ultrafilter is rapid.
\end{rem}

\begin{lemma}
There exists a rapid maximal a-filter, if we assume the continuum hypothesis.
\end{lemma}
\begin{proof}
Let $\mathcal U$ be a rapid free ultrafilter on $\N$. Let
\[\Filter:=\big\{S\in\ClosedCharsets: \{n\in\N: 1/n\in S\}\in\mathcal U\big\}.\]
From the fact that $\mathcal U$ is a filter, it is straightforward to check that $\Filter$ is an a-filter. From the fact that $\mathcal U$ is rapid, resp.\ maximal, it is straightforward to check that $\Filter$ is a rapid, resp.\ prime a-filter. By theorem \ref{max-filter}, it suffices to show that $\Filter$ closed. Let $S\in\clos \Filter$. As $S$ is a closed set, there exists a closed $T\succ S$ such that $\{n\in\N: 1/n\in T\} = \{n\in\N: 1/n\in S\}$. Since $T\in\Filter$, $\{n\in\N: 1/n\in T\} \in \mathcal U$. Hence also $S\in\Filter$.
\end{proof}

\begin{prop}
There exists a prime ideal in $\GenKc$ which is both minimal and maximal, if we assume the continuum hypothesis.
\end{prop}
\begin{proof}
Let $\Filter$ be a rapid maximal a-filter. By theorem \ref{rapid-filter-closed-ideal}, $I(\Filter)$ is closed, hence $I(\Filter)$ is both a minimal and maximal prime ideal by corollary \ref{char-min-max-prime-ideals}.
\end{proof}

\section{$z$-ideals}
As the notion of $z$-ideal in the ring $\Cnt(X)$ of continuous functions on a topological space $X$ can be expressed by a purely algebraic condition \cite[4A]{GJ}, G.~Mason \cite{Mason73} used this condition to define a $z$-ideal of any commutative ring $R$ with $1$. 
\begin{df}
Denoting by $\Max(a)=\{M$ max.\ ideals of $R: a\in M\}$, $I\ideal R$ is a $z$-ideal if
\[(\forall a\in R)(\forall b\in I)(\Max(a) = \Max(b)\implies a\in I).\]
\end{df}
We proceed to show a similar characterization as for $z$-ideals in $\GenK$. As in \cite{HVIdeals}, we denote $Z(a):=\{S\in\Charsets: \restr a S = 0\}$.

\begin{thm}\label{thm_z-ideals_equiv}
Let $a, b\in\GenKc$. Then $\Max(a)\subseteq \Max(b)\iff Z(a)\subseteq Z(b)$.
\end{thm}
\begin{proof}
$\Rightarrow$: let $S\in Z(a)\setminus Z(b)$, i.e., $\restr a S=0$ and $\restr b S\ne 0$. By lemma \ref{inv-char}, there exists $T\in\Charsets$ with $T\subseteq S$ such that $\restr b T$ is invertible. Let $M$ be a maximal ideal containing $I:=\{x\in \GenKc: \restr x T=0\}\idealproper \GenKc$. Since $\restr a S = 0$, also $\restr a T = 0$, hence $a\in M$. Suppose that $b\in M$. Since $\restr b T$ is invertible, $\restr b U$ is invertible for some $U\succ T$. By Urysohn's lemma, there exists $x\in\GenKc$ such that $\restr x T=0$ and $\restr x {\co U}= 1$. Hence $x\in I\subseteq M$, and $\bar x x + \bar b b = \abs x^2 + \abs b^2\in M$ would be invertible, a contradiction. We conclude that $M\in \Max(a)\setminus \Max(b)$.\\
$\Leftarrow$: let $M\in \Max(a)\setminus \Max(b)$, so $a\in M$ and $b\notin M$. As $M$ is maximal, $M + b\GenKc=\GenKc$. Let $m\in M$ and $c\in \GenKc$ such that $m+bc=1$. As $bc,m\in\GenK$, there exists $S\subseteq (0,1]$ such that $\restr{(bc)} S$ and $\restr m {\co S}$ are invertible \cite[Lemma 4.1]{HVIdeals}. Hence also $\restr b S$ is invertible. Suppose that $\restr a S$ is invertible. Then $\bar a a + \bar m m= \abs a^2 + \abs m^2 \in M$ would be invertible, a contradiction. By lemma \ref{inv-char}, there exists $T\in\Charsets$ with $T\subseteq S$ such that $\restr a T=0$. We conclude that $T\in Z(a)\setminus Z(b)$.
\end{proof}
\begin{cor}
$I\ideal \GenKc$ is a $z$-ideal iff
\[(\forall a\in \GenKc)(\forall b\in I)(Z(a) = Z(b)\implies a\in I).\]
\end{cor}

\begin{prop}\leavevmode
\begin{enumerate}
\item For $I\ideal\GenKc$,
\begin{align*}
\zclos{I}:&=\{x\in\GenKc: (\exists a\in I)(Z(x)= Z(a))\} =\{x\in\GenKc: (\exists a\in I)(Z(x)\supseteq Z(a))\}\\
&=\{x\in\GenKc: (\exists a\in I)(\Max(x)= \Max(a))\}
=\{x\in\GenKc: (\exists a\in I)(\Max(x)\supseteq \Max(a))\}
\end{align*}
is the smallest $z$-ideal containing $I$. We call it the $z$-closure of $I$. $I$ is a $z$-ideal iff $I=\zclos I$.
\item For $I\ideal\GenKc$, $I\subseteq\rad I\subseteq\zclos{I}$. Hence $\zclos{(\rad I)}=\zclos I$ and every $z$-ideal is radical.\\
A (proper) $z$-ideal is prime iff it is pseudoprime.
\end{enumerate}
\end{prop}
\begin{proof}
As in \cite[Prop.\ 4.3]{HVIdeals}.
\end{proof}

\begin{prop}\label{closed-is-z}
Every closed ideal $I\idealproper \GenKc$ is a $z$-ideal.
\end{prop}
\begin{proof}
$I$ is an intersection of maximal ideals (proposition \ref{intersection-of-max-ideals}), hence a $z$-ideal \cite{Mason73}.
\end{proof}

\begin{prop}\label{z-part}\leavevmode
\begin{enumerate}
\item For a family $(I_\lambda)_{\lambda\in\Lambda}$ of ideals $I_\lambda\ideal\GenKc$, $\zclos{(\sum_{\lambda\in\Lambda} I_\lambda)}=\sum_{\lambda\in\Lambda} \zclos{(I_\lambda)}$. In particular, the sum of a family of $z$-ideals is a $z$-ideal.
\item For $I,J\ideal\GenKc$, $\zclos{I}\cap \zclos{J}=\zclos{(I\cap J)}$.
\item For $I\ideal\GenKc$, $\zpart I:=\{x\in\GenKc: \zclos{(x\GenKc)}\subseteq I\}$ is the largest $z$-ideal contained in $I$. We call it the $z$-part of $I$. $I$ is a $z$-ideal iff $I=\zpart I$.
\item For a family $(I_\lambda)_{\lambda\in\Lambda}$ of ideals $I_\lambda\ideal\GenKc$, $\bigcap_{\lambda\in\Lambda} \zpart I_\lambda=\zpart{(\bigcap_{\lambda\in\Lambda} I_\lambda)}$. In particular, the intersection of a family of $z$-ideals is a $z$-ideal.
\item For $I\ideal\GenKc$, $m(I)\subseteq\zpart I\subseteq \radpart I\subseteq I$. In particular, every pure ideal of $\GenKc$ is a $z$-ideal. If $I\idealproper\GenKc$ is pseudoprime, then $\zpart I$ is prime.
\end{enumerate}
\end{prop}
\begin{proof}
1. First, we show that $\zclos{(I+J)}=\zclos I + \zclos J$.\\
Let $x\in \zclos{(I+J)}$. Hence there exist $a\in I$, $b\in J$ such that $Z(x)=Z(a+b)$. Let $(\alpha_\eps)_\eps$, resp.\ $(\beta_\eps)_\eps$, be representatives of $\abs a$, resp.\ $\abs b$, with $\alpha_\eps\ne 0$ and $\beta_\eps\ne 0$ for all $\eps$. Let $S:=\{\eps\in (0,1]: \alpha_\eps< 2\beta_\eps\}$ and $T:=\{\eps\in (0,1]: \beta_\eps< 2\alpha_\eps\}$. As $\alpha_\eps\ne 0$ and $\beta_\eps\ne 0$, $S\cup T=(0,1]$. By lemma \ref{prec-union}, there exist $V\prec S$, $U\prec T$ such that $U\cup V=(0,1]$. By Urysohn's lemma, there exists $y,z\in\GenRc$ such that $\restr y V = 1$, $\restr y {\co S} = 0$, $\restr z U = 1$ and $\restr z {\co T} = 0$ and $0\le y,z\le 1$. Then $y+ z\ge 1$. Hence there exists $u\in\GenRc$ such that $(y+z)u=1$.\\
Now let $W\in Z(a)$, i.e., $\restr a W = 0$. As $\restr{\abs b} T\le 2\restr{\abs a}T$, also $\restr b {T\cap W} = 0$. Hence $T\cap W\in Z(a+b)=Z(x)$, i.e.\ $\restr x {T\cap W} = 0$. Hence $\restr{xzu} {W} = \restr{xzu} {(W\cap T) \cup (W\setminus T)}=0$. Thus $Z(a)\subseteq Z(xzu)$. As $a\in I$, $xzu\in \zclos I$. Similarly, $xyu\in \zclos J$. Hence $x= xyu + xzu \in \zclos I + \zclos J$.\\
For arbitrary sums, the result follows as in \cite[Prop.\ 4.4]{HVIdeals}.

2--4. As in \cite[Prop.\ 4.4]{HVIdeals}.

5. We show that $\pure(I)\subseteq \zpart I$. Let $x\in \pure(I)=I(\Filter(I))$. Then there exists $S\in \Filter(I)$ such that $\restr x S = 0$. Let $y\in \zclos{(x\GenKc)}$. Then also $\restr y S = 0$, so $y\in I(\Filter(I)) \subseteq I$. Thus $\zclos{(x\GenKc)}\subseteq I$. The other statements follow as in \cite[Prop.\ 4.4]{HVIdeals} (using \cite[Prop.\ 4.29]{BK}).
\end{proof}
\begin{rem}
There are $z$-ideals that are not closed (e.g., consider a minimal prime ideal that is not maximal).
\end{rem}

It is well known that $\GenK$ is complete for the sharp topology \cite{Scarpa93}. Similarly, we have:
\begin{thm}
$\GenKc$ is complete for the sharp topology.
\end{thm}
\begin{proof}
Since $\GenKc\subseteq \GenK$ and $\GenK$ is complete, we show that $\GenKc$ is closed in $\GenK$.
Let $x_n\in\GenKc$ with continuous representative $(x_{n,\eps})_\eps$ such that $x_n\to x\in\GenK$. By taking a subsequence, we may assume that for each $n\in\N$,
\[\abs{x_{n,\eps}-x_{\eps}}\le\eps^n,\quad \forall \eps\le\eps_n.\]
W.l.o.g., $(\eps_n)_n$ is strictly decreasing and tends to $0$. Then let $u_{1,\eps}:= x_{1,\eps}$ and
\[
u_{n,\eps}:= 
\begin{cases}
x_{n,\eps} - x_{n-1,\eps}&\eps\le \eps_{n+1}\\
0,& \eps >\eps_n
\end{cases}
\]
in such a way that $u_{n,\eps}$ is continuous in $\eps$ and $\abs{u_{n,\eps}}\le \abs{x_{n,\eps} - x_{n-1,\eps}}$ for each $\eps\in (0,1]$. Then $s_\eps := \sum_{n=1}^\infty u_{n,\eps}$ is a locally finite sum. Hence $(s_\eps)_\eps$ is continuous and for each $\eps\in (\eps_{n+1},\eps_n]$,
\[
\abs{s_\eps-x_\eps} = \abs[\Big]{\sum_{k=1}^n u_{k,\eps} - x_\eps} \le \abs{u_{n,\eps}} +\abs {x_{n-1,\eps}-x_\eps}\le \abs{x_{n,\eps}-x_\eps} + 2\abs{x_{n-1,\eps}-x_\eps}\le 3\eps^{n-1}.
\]
Hence $x=[s_\eps]\in\GenKc$.
\end{proof}

\begin{thm}\label{thm_closed_fin_gen}
Let $I\ideal\GenKc$ be a finitely generated ideal.
\begin{enumerate}
\item If $I$ is radical (in particular, if $I$ is closed, pure or a $z$-ideal), then $I\in\{0,\GenKc\}$.
\item $\zclos I=\clos I$
\item $m(I)=\zpart I$.
\end{enumerate}
\end{thm}
\begin{proof}
By \cite[Lemma 4.5]{BK}, $I$ is principal, i.e.\ $I=a\GenKc$ for some $a\in\GenKc$.

1. By \cite[Prop.\ 4.28]{BK}, $I$ is idempotent. Hence $a=a^2 b$ for some $b\in\GenKc$. Thus $ab$ is idempotent. So either $ab=0$, whence $a=a^2 b = 0$ and $I=0$, or $ab=1$, whence $I = \GenKc$.

2. Let $x\in \clos I$, i.e.\ $x=\lim_{n\to\infty} x_n$ for some $x_n\in I$. Let $S\in Z(a)$, i.e.\ $\restr a S=0$. Then also $\restr{x_n}S=0$ for each $n\in\N$, hence also $\restr x S=0$, i.e.\ $S\in Z(x)$. Thus $x\in \zclos I$. The converse inclusion holds by proposition \ref{closed-is-z}.

3. Let $x\in \GenKc\setminus \pure(I) = I(\Filter(I))$. Then for each $S\in\Filter(I)$, $\restr x S\ne 0$. In particular, let $(a_\eps)_\eps$ be a (continuous) representative of $a$ and $L_n:=\{\eps\in (0,1]: \abs{a_\eps}> \eps^n\}$. Then $\co L_n\in\Filter(I)$, so $\restr x {\co L_n}\ne 0$. By lemma \ref{inv-char}, there exist $T_n\in\Charsets$ with $T_n\subseteq \co L_{n+1}$ and $\restr x {T_n}$ is invertible. By lemma \ref{zero-inv-ext}, there exist $S_n\succ T_n$ such that $\restr x {S_n}$ is invertible (as $T_n\prec \co L_n$, we may assume $S_n\prec\co L_n$). By Urysohn's lemma, there exist $y_n\in\GenKc$ with $\restr{y_n} {T_n} = \restr{(\sqrt{\abs a})}{T_n}$, $\restr {y_n}{\co S_n} = 0$ and $0\le y_n\le \sqrt{\abs a}$. As $\restr{\abs{y_n}}{S_n}\le \restr{\sqrt{\abs a}}{S_n}\le\caninf^{n/2}$, $y:=\sum_{n=1}^\infty y_n\in\GenKc$ exists ($\GenKc$ is a complete ultrametric space). We show that $y\in \zclos{(x\GenKc)}$.\\
Let $U\in Z(x)$, i.e., $\restr x U=0$. Then $0\notin \clos{U\cap S_n}$, since $\restr x {S_n}$ is invertible. Hence $\restr{y_n} U = 0$. Then also $\restr y U = 0$, i.e., $U\in Z(y)$.\\
Also $y\notin I$: $\restr{\abs y} {T_n}\ge \restr{\abs{y_n}} {T_n} = \restr{\sqrt{\abs a}}{T_n}\ge \restr{(\caninf^{-n/2}\abs a )}{T_n}$ for each $n\in\N$. Hence $\abs{y}\nleq \caninf^{-N}\abs a$ for any $N\in\N$, and thus $y\notin I$.
Hence $\zclos{(x\GenKc)}\not\subseteq I$, i.e., $x\notin \zpart I$.
\end{proof}

Let $I\idealproper\GenKc$. Let $I^\perp=\{x\in\GenKc: xy=0, \forall y\in I\}$. As in $\GenK$, we have:
\begin{prop}
Let $I\idealproper \GenKc$. Then
\begin{enumerate}
\item $I^\perp$ is closed.
\item $\clos I \subseteq I^{\perp\perp}$.
\item $\clos I\cap I^\perp=\{0\}$.
\item If $I$ is pseudoprime, then $I^\perp=\{0\}$. In particular, $\clos I\subsetneqq I^{\perp\perp} = \GenKc$.
\end{enumerate}
\end{prop}
\begin{proof}
1. Let $x=\lim_{n\to\infty} x_n$, with $x_n\in I^\perp$. Then $x_n y=0$, $\forall n\in\N$, hence also $xy=0$, $\forall y\in I$. Thus $x\in I^\perp$.

2. If $x\in I$, then $xy=0$, $\forall y\in I^\perp$, so $I \subseteq I^{\perp\perp}$. By part 1, also $\clos I\subseteq I^{\perp\perp}$.

3. If $x\in I\cap I^\perp$, then $x^2=0$, hence $x=0$. Hence also $I^\perp\cap \clos I\subseteq I^\perp\cap I^{\perp\perp}=\{0\}$.

4. Let $x\in I^\perp$. If $x\ne 0$, then there exists $T\in\Charsets$ such that $\restr x T$ is invertible. By lemma \ref{zero-inv-ext}, there exists $S\succ T$ such that $\restr x S$ is invertible. W.l.o.g.\ $S$ is closed, $T$ is open and $\co S\in\Charsets$. As $\inter{(\co T)} \cup\inter S= (0,1]$ and $\Filter(I)$ is pseudoprime, either $\co{T} \in\Filter(I)$ or $S\in\Filter(I)$. In the first case, there exists $y\in I$ such that $\restr{y}{ T}=1$. As $x\in I^\perp$, $xy=0$, contradicting the fact that $\restr{(xy)}{T}$ is invertible.
In the second case, there exists $y\in I$ such that $\restr y {\co S} = 1$. Hence $xy=0$, and thus $\restr x{\co S} = 0$. As $\restr {(xz)} S=1$ for some $z\in\GenKc$, and $\restr{(xz)}{\co S}=0$, $xz\in\GenKc$ is idempotent, and hence $xz=0$ (contradicting $\restr{(xz)}S=1$) or $xz=1$ (contradicting $\restr{(xz)}{\co S}=0$).  Thus $x=0$.
\end{proof}

\begin{lemma}\label{HB-lemma}
There exists $J\idealproper\GenKc$ such that $J\ne \{0\}$ and $J^\perp\ne\{0\}$.
\end{lemma}
\begin{proof}
Let $S := \bigcup_{n\in\N} (a_n,b_n)$ with $1>b_1>a_1>b_2>a_2>\dots$ and $a_n\to 0$. Then there exists $x\in \GenKc\setminus\{0\}$ such that $\restr x S=0$ and there exists $y\in\GenKc\setminus\{0\}$ such that $\restr y {\co S}=0$. Let $J=\{x\in\GenKc: \restr x S =0\}$. Then $x\in J$ and $y\in J^\perp$. 
\end{proof}

Also as in $\GenK$, the Hahn-Banach extension property does not hold in the following sense:
\begin{thm}
Let $J\idealproper\GenKc$ with $J\ne \{0\}$ and $J^\perp\ne\{0\}$. Let $I:= J+J^\perp$. Then there exists a continuous $\GenKc$-linear map $\phi$: $I\to\GenKc$ that cannot be extended to a $\GenKc$-linear map $\psi$: $\GenKc\to\GenKc$.
\end{thm}
\begin{proof}
Let $\phi(x+y):= x$, for each $x\in J$ and $y\in J^\perp$. As $J\cap J^\perp=\{0\}$, $\phi$ is defined unambiguously and is $\GenKc$-linear. Further, $\abs{\phi(x+y)}^2 = \abs x^2 \le \abs x^2 + \abs y^2 = (x+y)(\bar x + \bar y) = \abs{x+y}^2$, for each $x\in J$ and $y\in J^\perp$. Hence $\phi$ is also continuous.

Now suppose that $\psi$: $\GenKc\to\GenKc$ is a $\GenKc$-linear extension of $\phi$. Then for any $x\in J$, $x\psi(1) = \psi(x) =\phi(x) =x$. Hence $x(\psi(1)-1)=0$. Thus $\psi(1)-1\in J^\perp$. Hence $\psi(1)\psi(1)-\psi(1)=\psi(\psi(1)-1)=\phi(\psi(1)-1)=0$. It follows that $\psi(1)\in\GenKc$ is idempotent, hence $\psi(1)=0$ or $\psi(1)=1$. If $\psi(1)=0$, then $\psi=0$, and thus also $\phi=0$, whence $J=\{0\}$. If $\psi(1)=1$, then $\psi(x)=x$ for each $x\in\GenKc$, and thus also $\phi(y)=y$ for each $y\in J^\perp$, whence $J^\perp=\{0\}$.
\end{proof}

\begin{cor}
If $I\idealproper\GenKc$ with $I\ne\{0\}$ and $I^\perp\ne\{0\}$, then $I+I^\perp\ne\GenKc$.
\end{cor}

\end{document}